\documentclass[oneside,english,american]{amsart}
\usepackage[T1]{fontenc}
\usepackage[utf8]{inputenc}
\usepackage{amstext}
\usepackage{amsthm}
\usepackage{amssymb}

\makeatletter
\numberwithin{equation}{section}
\numberwithin{figure}{section}
\theoremstyle{plain}
\newtheorem{thm}{\protect\theoremname}
\theoremstyle{plain}
\newtheorem{lem}[thm]{\protect\lemmaname}

\makeatother

\usepackage{babel}
\addto\captionsamerican{\renewcommand{\lemmaname}{Lemma}}
\addto\captionsamerican{\renewcommand{\theoremname}{Theorem}}
\addto\captionsenglish{\renewcommand{\lemmaname}{Lemma}}
\addto\captionsenglish{\renewcommand{\theoremname}{Theorem}}
\providecommand{\lemmaname}{Lemma}
\providecommand{\theoremname}{Theorem}

\begin{document}
\title{Isomorphic Subtypes in a Finite Generalized Ordered Type}
\author{Jean S. Joseph}

\maketitle
In what follows, we answer the following: if \foreignlanguage{english}{$G$
is a finite generalized ordered type and $S$ a subtype of $G$, is
there an algorithm that finds all subtypes of $G$ isomorphic to $S$? }
\selectlanguage{english}%

\section{Generalized Ordered Types}

Let $G$ be a type and let $<:G\rightarrow G\rightarrow\mathsf{Prop}$,
and let $x\leq y$ be defined as, for all $z:G$, $z<x$ implies $z<y$
and $y<z$ implies $x<z$.

The type $G$ is a \textbf{generalized ordered type} if, for all $x,y,z:G$,
\begin{enumerate}
\item $x<y\rightarrow\neg y<x$
\item $x<y<z\rightarrow x<z$
\item $x\leq y$ and $y\leq x\rightarrow x=y$.
\end{enumerate}
\begin{lem}
Let $G$ be a generalized ordered type. Then, for all $x:G$, $\neg x<x$.
\label{lem: 1-1}
\end{lem}

\begin{proof}
Note that $x<x$ implies $\neg x<x$ by asymmetry. Suppose $x<x$;
then $\neg x<x$, which is impossible.
\end{proof}
\begin{lem}
Let $G$ be a generalized ordered type. For all $x,y:G$, $x=y$ implies
$x\leq y$ and $y\leq x$. \label{lem: 1}
\end{lem}

\begin{proof}
Note that for all $x:G$, $x\leq x$, so the result follows by path
induction.
\end{proof}
\begin{lem}
Let $G$ be a generalized ordered type. For all $x,y,z:G$, $x<y\leq z$
implies $x<z$, and $x\leq y<z$ implies $x<z$. \label{lem: 2}
\end{lem}

\begin{proof}
Since $y\leq z$, it follows $x<z$ by definition of $\leq$. 
\end{proof}
Two elements $x,y$ of $G$ are \textbf{different} if $\neg x=y$.
We write $x\neq y$ to mean $x,y$ are different. Two elements $x,y$
of $G$ are \textbf{unordered} if $\neg x<y$ and $\neg y<x$. A generalized
ordered type is \textbf{unbounded} if, for each $x:G$, there are
$y,y':G$ such that $y<x<y'$. A generalized ordered type $G$ is
\textbf{finite} if there is a positive integer $N$ and an equivalence
between $\mathsf{Fin}\left(N\right)$ and $G$.

An example of a generalized ordered type such that there exist $x,y$
with $\neg x<y$ and $\neg y<x$, but $\neg x=y$. Let $X$ be any
set and $Y$ any unbounded generalized ordered set, and let $\mathfrak{F}$
be the set of functions from subsingletons\footnote{A subsingleton is a subset of a singleton.}
of $X$ to $Y$ with $f=g$ if $\textrm{dom}f=\textrm{dom}g$ and
$f\left(x\right)=g\left(x\right)$ for all $x$ and $f<g$ if $\textrm{dom}f\subseteq\textrm{dom}g$
and\footnote{By $A\subseteq B$ we mean, for each $x:A$, there is $b_{x}:B$ such
that $x=b_{x}$.} $f\left(x\right)<g\left(x\right)$ for all $x:\textrm{dom}f\cap\textrm{dom}g$.
Suppose $f<g$ and $g<f$, so $\textrm{dom}f\subseteq\textrm{dom}g$
and $\textrm{dom}g\subseteq\textrm{dom}f$, implying $\textrm{dom}f=\textrm{dom}g$.
Let $x:\textrm{dom}f$; then $f\left(x\right)<g\left(x\right)$ since
$f<g$ and $g\left(x\right)<f\left(x\right)$ since $g<f$, but this
is impossible since the order on $Y$ is asymmetric. Hence, the order
on $\mathfrak{F}$ is asymmetric. Now, suppose $f<g<h$. Since $\textrm{dom}f\subseteq\textrm{dom}g\subseteq\textrm{dom}h$,
it follows $\textrm{dom}f\subseteq\textrm{dom}h$. Let $x:\textrm{dom}f\cap\textrm{dom}h$.
Since $\textrm{dom}f\subseteq\textrm{dom}g$, it follows $x:\textrm{dom}g$,
so $f\left(x\right)<g\left(x\right)<h\left(x\right)$; hence, $f\left(x\right)<g\left(x\right)$
since the order on $Y$ is transitive. Thus the order on $\mathfrak{F}$
is transitive. 
\begin{lem}
\textup{For all $f,g:\mathfrak{F}$, if $\textrm{dom}f\subseteq\textrm{dom}g$
and, for each $x:\textrm{dom}f\cap\textrm{dom}g$, $f\left(x\right)\leq g\left(x\right)$,
then $f\leq g$. \label{lem: 3}}
\end{lem}

\begin{proof}
Let $h<f$. Since $\textrm{dom}h\subseteq\textrm{dom}f\subseteq\textrm{dom}g$,
it follows $\textrm{dom}h\subseteq\textrm{dom}g$. If $x:\textrm{dom}h\cap\textrm{dom}g$,
then $h\left(x\right)<f\left(x\right)$ since $x:\textrm{dom}f$ and
$f\left(x\right)\leq g\left(x\right)$ since $x:\textrm{dom}f\cap\textrm{dom}g$,
so $h\left(x\right)<g\left(x\right)$ by Lemma \ref{lem: 2}. Hence,
$h<g$. 

Now let $g<k$. Since $\textrm{dom}f\subseteq\textrm{dom}g\subseteq\textrm{dom}k$,
it follows $\textrm{dom}f\subseteq\textrm{dom}k$. Let $x:\textrm{dom}f\cap\textrm{dom}k$;
then $f\left(x\right)\leq g\left(x\right)$ by hypothesis and $g\left(x\right)<k\left(x\right)$,
so $f\left(x\right)<k\left(x\right)$ by Lemma \ref{lem: 2}. Hence,
$f<k$.
\end{proof}
For positive antisymmetry, suppose $f\leq g$ and $g\leq f$. Let
$x:\textrm{dom}f$. Since $Y$ is unbounded, there is $y:Y$ such
that $y<x$. Let $h$ be the function from $\left\{ x\right\} $ to
$\left\{ y\right\} $ defined as $h\left(x\right)=y$, so $h<f$;
thus $h<g$, so $x:\textrm{dom}g$. Therefore, $\textrm{dom}f\subseteq\textrm{dom}g$.
Similarly, $\textrm{dom}g\subseteq\textrm{dom}f$. Hence, $\textrm{dom}f=\textrm{dom}g$.
Now let $x:\textrm{dom}f$, and let $y<f\left(x\right)$. We define
the function $h:\left\{ x\right\} \rightarrow\left\{ y\right\} $
defined as $h\left(x\right)=y$ and the function $j:\left\{ x\right\} \rightarrow\left\{ f\left(x\right)\right\} $,
so $h<j$. It follows $j\leq f$ by Lemmas \ref{lem: 1} and \ref{lem: 3},
so $h<f$, implying $h<g$ since $f\leq g$; thus $y<g\left(x\right)$.
Now, let $g\left(x\right)<y'$ and let $h':\left\{ x\right\} \rightarrow\left\{ y'\right\} $
be the function defined as $h'\left(x\right)=y'$, so $g<h'$, implying
$f<h'$; hence, $f\left(x\right)<y'$. Therefore, $f\left(x\right)\leq g\left(x\right)$.
Similarly, $g\left(x\right)\leq f\left(x\right)$. Therefore, $f\left(x\right)=g\left(x\right)$. 

Let $X=\left\{ 0,1\right\} \subseteq\mathbb{Z}$ and $Y=\mathbb{Z}$,
and, for any proposition $P$, let $p$ be any constant function defined
on $\left\{ 0:P\right\} $ and $q$ any constant function defined
on $\left\{ 1:\neg P\right\} $. Note that $\neg p<q$ because $p<q$
implies $\textrm{dom}p\subseteq\textrm{dom}q$, which is false, and
note also that $\neg q<p$. Moreover, $p=q$ implies $\textrm{dom}p=\textrm{dom}q$,
so $\textrm{dom}p$ is inhabited if and only if $\textrm{dom}q$ is
inhabited, which is false by definition of the domains of $p$ and
$q$. Therefore, $\neg p=q$. 

We write $x\parallel y$ to mean $x,y$ are unordered. A generalized
ordered type $G$ is \textbf{weakly discrete} if, for all $x,y:G$,
$x\neq y$ implies $x<y$, $y<x$, or $x\parallel y$. Two generalized
ordered types $G,G':\mathcal{U}$ are \textbf{isomorphic} if there
is an equivalence $f:G\rightarrow G'$ such that, for all $x,y:G$,
$x<y$ if and only if $f\left(x\right)<f\left(y\right)$.
\begin{lem}
Let $X,Y:\mathcal{U}$. If $X$ is finite and $X\simeq Y$, then $Y$
is finite. \label{lem: 5}
\end{lem}

-
\begin{lem}
Let $X,Y:\mathcal{U}$. If $X,Y$ are finite, then $X+Y$ is finite.
\label{lem: 6}
\end{lem}

\begin{proof}
There are positive integers $N,M$ such that $\mathsf{Fin}\left(N\right)\simeq X$
and $\mathsf{Fin}\left(M\right)\simeq Y$, so $\mathsf{Fin}\left(N\right)+\mathsf{Fin}\left(M\right)\simeq X+Y$.
Let $f:\mathsf{Fin}\left(N\right)+\mathsf{Fin}\left(M\right)\rightarrow\mathsf{Fin}\left(N+M\right)$
be defined as $f\left(i\right)=i$ if $i:\mathsf{Fin}\left(N\right)$
and $f\left(i\right)=N+i$ if $i:\mathsf{Fin}\left(M\right)$, and
let $g:\mathsf{Fin}\left(N+M\right)\rightarrow\mathsf{Fin}\left(N\right)+\mathsf{Fin}\left(M\right)$
be defined as $g\left(i\right)=i$ if $i<N$ and $g\left(i\right)=i-N$
if $N\leq i$. Hence, $\mathsf{Fin}\left(N\right)+\mathsf{Fin}\left(M\right)\simeq\mathsf{Fin}\left(N+M\right)$.
Therefore, $\mathsf{Fin}\left(N+M\right)\simeq X+Y$.
\end{proof}
For any natural number $n\geq2$, we write $X_{0}+\cdots+X_{n}$ to
mean $\left(X_{0}+\cdots+X_{n-1}\right)+X_{n}$.

-
\begin{lem}
For any natural number $n$, if $X_{0},\ldots,X_{n-1}:\mathcal{U}$
are finite types, then $X_{0}+\cdots+X_{n-1}$ is finite. \label{lem: 7}
\end{lem}

\begin{proof}
We proceed by induction on $n$. Suppose the lemma is true for any
natural number less than $n$, so $X_{0}+\cdots+X_{n-2}$ is finite.
Hence, $\left(X_{0}+\cdots+X_{n-2}\right)+X_{n-1}$ is finite by Lemma
\ref{lem: 6}.
\end{proof}
-
\begin{lem}
If $X,Y:\mathcal{U}$ are finite types, then $X\times Y$ is finite.
\label{lem: 8}
\end{lem}

\begin{proof}
There are positive integers $M,N$ such that $\mathsf{Fin}\left(N\right)\simeq X$
and $\mathsf{Fin}\left(M\right)\simeq Y$. Consider the family $L:\mathsf{Fin}\left(N\right)\rightarrow\mathcal{U}$
defined by $L\left(k\right):=\left\{ x_{k}\right\} \times\left\{ y_{0},\ldots,y_{M-1}\right\} $.
Since $\mathsf{Fin}\left(M\right)\simeq Y\simeq L\left(k\right)$
for each $k$, it follows $L\left(k\right)$ is finite for each $k$.
Note that, for each $k$, each term of $L\left(k\right)$ is a term
of $X\times Y$, and each $\left(x_{n},y_{m}\right):X\times Y$ is
a term of $L\left(n\right)$. Hence, $L\left(0\right)+\cdots+L\left(n-1\right)\simeq X\times Y$.
Therefore, $X\times Y$ is finite by Lemmas \ref{lem: 5} and \ref{lem: 7}. 
\end{proof}
For types $X_{0},\ldots,X_{n-1}:\mathcal{U}$, we write $X_{0}\times\cdots\times X_{n-1}$
to mean $\left(X_{0}\times\cdots\times X_{n-2}\right)\times X_{n-1}$.
\begin{lem}
For any natural number $n$, if $X_{0},\ldots,X_{n-1}:\mathcal{U}$
are finite types, then $X_{0}\times\cdots\times X_{n-1}$ is finite.
\label{lem: 9}
\end{lem}

\begin{proof}
We proceed by induction on $n$. Assume the lemma is true for any
natural number less than $n$. Then $X_{0}\times\cdots\times X_{n-2}$
is finite. Hence, $X_{0}\times\cdots\times X_{n-1}$ is finite by
Lemma \ref{lem: 8}.
\end{proof}

\section{The Theorem}

We now can state and prove the theorem:
\begin{thm}
Let $G$ be a finite, weakly discrete generalized ordered type, and
let $S$ be the subtype of $G$, consisting of two elements $a,b$
with $a<b$. Then there is a terminating algorithm that finds all
subtypes of $G$ isomorphic to $S$.
\end{thm}

\begin{proof}
Let $x_{0},\ldots,x_{N-1}$ be the elements of $G$. Since $G\times G$
is finite by Lemma \ref{lem: 8}, consider all pairs $\left(x_{i},x_{j}\right)$
with $i,j:\mathsf{Fin}\left(N\right)$. If $i=j$, then $x_{i}=x_{j}$,
so $\neg x_{i}<x_{j}$ and $\neg x_{j}<x_{i}$ by Lemma \ref{lem: 1-1}.
Hence, the subtype $\left\{ x_{i},x_{j}\right\} $ is not isomorphic
to $S$. If $i\neq j$, then $x_{i}\neq x_{j}$. If $x_{i}\parallel y_{j}$,
then there is no isomorphism between $\left\{ x_{i},x_{j}\right\} $
and $S$. If $x_{i}<x_{j}$, then the equivalence $\begin{array}{c}
a\mapsto x_{i}\\
b\mapsto x_{j}
\end{array}$ is an isomorphism, and if $x_{j}<x_{i}$, then the equivalence $\begin{array}{c}
a\mapsto x_{j}\\
b\mapsto x_{i}
\end{array}$ is an isomorphism. 
\end{proof}
\selectlanguage{american}%

\end{document}